\documentclass[12pt]{amsart}
\usepackage{amscd,amsthm,amsfonts,amssymb,amsmath,euscript}
\usepackage[T2A]{fontenc}
\usepackage[cp866]{inputenc}

\newtheorem{theorem}{Theorem}[section]
\newtheorem{lemma}{Lemma}[section]
\newtheorem{corollary}{Collorary}[section]
\newtheorem{example}{Example}[section]

\newcommand{\Aut}{\mathop{\sf Aut}\nolimits}

\newcommand{\tr}{\mathop{\sf tr}\nolimits}

\begin{document}

\begin{center}

\hfill FIAN/TD-11/10\\
\hfill ITEP/TH-46/10

\vspace{1cm}

\end{center}
\centerline  {\large\bf Algebra of differential operators associated with Young
diagrams}

{\ }

\centerline {A.Mironov, A.Morozov, S.Natanzon}
\address{Theory Department, Lebedev Physical Institute, Moscow, Russia;
Institute for Theoretical and Experimental Physics, Moscow, Russia}
\email{mironov@itep.ru; mironov@lpi.ru}
\address{Institute for Theoretical and Experimental Physics;
Laboratoire de Mathematiques et
Physique Theorique, CNRS-UMR 6083, Universite Francois Rabelais de
Tours, France}
\email{morozov@itep.ru}
\address {Department of Mathematics, Higher School of Economics, Moscow, Russia;
A.N.Belozersky Institute, Moscow State University, Russia;
Institute for Theoretical and Experimental Physics}
\email{natanzons@mail.ru}

\vspace{.5cm}

{\ }

\rightline {\it {To the memory of Vladimir Igorevich Arnold}}

\vspace{1cm}

{\footnotesize
We establish a correspondence between Young diagrams and differential operators of infinitely
many variables. These operators form a commutative associative algebra isomorphic to the algebra of
the conjugated classes of finite permutations of the set of natural numbers. The Schur functions form a complete
system of common eigenfunctions of these differential operators, and their eigenvalues are expressed through
the characters of symmetric groups. The structure constants of the algebra are expressed through the Hurwitz numbers.
}

\bigskip

\tableofcontents

\section{{\bf Introduction}}

Center of the group algebra $A_n$ of the symmetric group $S_n$ plays the main role
in describing representations both of the symmetric group and of the matrix group
$Gl(n)$. Its counterpart for the infinite symmetric group is the algebra $A_{\infty}$
of the conjugated classes of finite permutations of an infinite set
\cite{IK}. Its natural generators are the Young diagrams of arbitrary degree.

In the present paper, which is a continuation of \cite{MMN1},
we construct an exact representation of algebra
$A_{\infty}$ in the algebra of differential operators
of infinitely many variables. The differential operators $\textsf{W}(\Delta)$,
corresponding to the Young diagrams $\Delta$, are closely related to the
Hurwitz numbers, matrix integrals and integrable systems
\cite{MS,BM,MM,M1,M2}. We prove that the Schur functions form a complete set of the
common eigenfunctions of the operators
$\textsf{W}(\Delta)$, and find the corresponding eigenvalues. A key role in the
construction
is played by the Miwa variables, which naturally emerge in matrix models
\cite{Miwa,GKM}.

In section 2, we define the algebra of Young diagrams, which is isomorphic to the
algebra of conjugated classes of finite permutations of an infinite set, and express
its structure constants through the structure constants of the algebra $A_n$.
In section 3, we construct a representation of the universal enveloping
algebra $U(gl(\infty))$ in the algebra of differential operators of Miwa variables.
Using this representation, in section 4, we associate with any Young diagram
a differential operator $\mathcal{W}(\Delta)$ of Miwa variables, which has a very simple form.
This correspondence gives rise to an exact representation of the algebra
$A_\infty$.

The operators $\mathcal{W}(\Delta)$ preserve the subspace
$\textsf{P}$ of all symmetric polynomials of the Miwa variables.
We study further the differential operators $\textsf{W}(\Delta)=
\mathcal{W}(\Delta)|_{\textsf{P}}$ of the variables $p=\{p_i\}$,
which form a natural basis in the space $\textsf{P}$. In section 5, we prove that the
Schur functions $s_R(p)$ form a complete system of eigenfunctions for
$\textsf{W}(\Delta)$ and find the corresponding eigenvalues.
In section 6, we explain an algorithm of calculating the operators
$\textsf{W}(\Delta)$, the simplest non-trivial operator $\textsf{W}([2])$
being nothing but the "cut-and-join" operator \cite{GJV}, which plays an important role
in the theory of Hurwitz numbers and moduli spaces.

In the last section 7, we interpret the operators
$\textsf{W}(\Delta)$ as counterparts of the "cut-and-join" operator for the arbitrary Young
diagram. In particular, we prove that a special generating function of Hurwitz numbers
satisfies a simple differential equation, which allows one to construct all the Hurwitz
numbers successively.

We thank M.Kazarian, A.N.Kirillov, S.Lando and S.Loktev for fruitful
discussions. S.Natanzon is also grateful to IHES for perfect conditions
for finishing this work.
Our work is partly supported by Ministry of Education and Science of
the Russian Federation under contract 14.740.11.0081, by RFBR
grants 10-02-00509-a (A.Mir. \& S.N.) and 10-02-00499 (A.Mor.),
by joint grants 09-02-90493-Ukr, 09-01-92440-CE, 09-02-91005-ANF,
10-02-92109-Yaf-a. The work of A.Morozov was also supported in part
by CNRS.

\section{{\bf Algebra $A_{\infty}$ of Young diagrams}}\label{s1}

1. First we remind the standard facts that we need below. Denote through
$|\mathfrak{M}|$ the number of elements in a finite set $\mathfrak{M}$
and through $S_n$ the symmetric group which acts by permutations on the set
$\mathfrak{M}$, where $|\mathfrak{M}|=n$. A permutation $g\in S_n$ gives rise to
a subgroup $<g>$, whose action divides
$\mathfrak{M}$ into orbits $\mathfrak{M}_1,\dots,\mathfrak{M}_k$. The set of numbers
$|\mathfrak{M}_1|,\dots,|\mathfrak{M}_k|$ is called \textit{cyclic type of the
permutation $g$}. It produces the Young diagram $\Delta(g)=[|\mathfrak{M}_1|,\dots,
|\mathfrak{M}_k|]$ of degree $n$. The permutations are conjugated in $S_n$ if and only
if they are of the same cyclic type.

Linear combinations of the permutations from $S_n$ form the group algebra
$G_n=G(S_n)$. Multiplication in this algebra is denoted as $"\circ"$.
Associate with each Young diagram $\Delta$ the sum $G_n(\Delta)\in G_n$
of all permutations of the cyclic type $\Delta$. These sums form the basis
\textit{of the algebra of the conjugated classes} $A_n^{\circ}\subset
G_n$, which coincides with the center $G_n$.

Denote through $C_{\Delta_1,\Delta_2}^{\Delta}$ the structure constants
of the algebra $A_n$ in this basis. In other words,
$$G_n(\Delta_1)\circ G_n(\Delta_2)=\sum\limits_ {\Delta\in\mathcal{A}_n}
C_{\Delta_1,\Delta_2}^{\Delta}G_n(\Delta),$$
where $\mathcal{A}_n$ is the set of all Young diagrams $\Delta$ of degree $|\Delta|=n$.

The construction of algebra $A_n^{\circ}$ is continued in \cite{IK} to the algebra
$A_\infty$ of the conjugated classes of finite permutations of the set of natural numbers
$\mathbb{N}=\{1,2,\dots\}$. The algebra $A_\infty$ is generated by
$G_{\infty}(\Delta)$, which are a formal sum of all finite permutations of the set
$\mathbb{N}$ of the cyclic type $\Delta$. Multiplication in the algebra is
generated by the multiplication of permutations. According to \cite{IK}, this
algebra is naturally isomorphic to the algebra of the shifted Schur functions
\cite{OO}.

2. We express now the structure constants of the algebra $A_\infty$ through the structure
constants of the algebra $A_n^{\circ}$. First, we represent the algebra $A_n^{\circ}$
as an algebra generated by Young diagrams. In other words, we consider
$A_n^{\circ}$ as a vector space with the basis $\mathcal{A}_n$ and multiplication
$$\Delta_1\circ \Delta_2=
\sum\limits_ {\Delta\in\mathcal{A}_n}C_{\Delta_1,\Delta_2}^{\Delta}\Delta.$$

Consider also the monomorphism of the vector space
$\rho_k:A_n^{\circ}\rightarrow A_{n+k}^{\circ}$,
where $\rho_k(\Delta)=\frac{(r+k)!}{r!k!}{\Delta}^k$.
Here ${\Delta}^k$ is the Young diagram obtained from the Young diagram
$\Delta$ by adding $k$ unit length rows and $r$ is the number of unit length
rows originally present in the diagram $\Delta$.

Let us define a multiplication of diagrams of arbitrary degree by the formula
$$\Delta_1\Delta_2=
\sum\limits_ {n=\max\{|\Delta_1|,|\Delta_2|\}} ^{|\Delta_1|+|\Delta_2|}
\{\Delta_1\Delta_2\}_{n}$$
where
$$
\{\Delta_1\Delta_2\}_n = \left\{\begin{array}{lr}\rho_{(n-|\Delta_1|)}(\Delta_1)\circ
\rho_{(n-|\Delta_2|)}(\Delta_2)&
\hbox{for } n=\max\{|\Delta_1|,|\Delta_2|\}\\ \\
\rho_{(n-|\Delta_1|)}(\Delta_1)\circ
\rho_{(n-|\Delta_2|)}(\Delta_2)&
-\sum\limits_ {k=\max\{|\Delta_1|,|\Delta_2|\}}^{n-1}
\rho_{(n-k)}(\{\Delta_1\Delta_2\}_k)\\ \\
&\hbox{for }n>\max\{|\Delta_1|,|\Delta_2|\}
\end{array}\right.$$

\begin{example} {\rm Put $\Delta_1=[1]$ and $\Delta_2=[2]$.
Then, $\rho_1([1])=2[1,1]$, $\rho_2([1])=3[1,1,1]$,
$\rho_1([2])=[2,1]$.
Therefore,
$\{\Delta_1\Delta_2\}_2=2[1,1]\circ[2]=2[2]$,
$\{\Delta_1\Delta_2\}_3=3[1,1,1]\circ[2,1]-2[2,1]=[2,1]$}
\end{example}

\begin{example} {\rm Put $\Delta_1=[2]$ and $\Delta_2=[2]$.
Then, $\rho_1([2])=[2,1]$, $\rho_2([2])=[2,1,1]$,
$\rho_1([1,1])=3[1,1,1]$, $\rho_2([1,1])=6[1,1,1,1]$,
$\rho_1([3])=[3,1]$.
Therefore,
$\{\Delta_1\Delta_2\}_2=[2]\circ[2]=[1,1]$,
$\{\Delta_1\Delta_2\}_3=[2,1]\circ[2,1]-3[1,1,1]=3[3]$ and
$\{\Delta_1\Delta_2\}_4=[2,1,1]\circ[2,1,1]-3[3,1]-6[1,1,1,1]=2[2,2]$.}
\end{example}

\begin{theorem} {\rm \label{t1} The operation $(\Delta_1,\Delta_2)\mapsto\Delta_1\Delta_2$
gives rise on $A_{\infty}^{\circ}=\bigoplus\limits_{n}A_n^{\circ}$ to the
structure of a commutative associative algebra.}
\end{theorem}

\begin{proof} Commutativity follows from the commutativity of the algebras
$A_n^{\circ}$. Associativity follows from the associativity of the algebras
$A_n^{\circ}$ and the equality

$$\Delta_1\Delta_2\Delta_3=\sum\limits_ {n=\max\{|\Delta_1|,|\Delta_2|,
|\Delta_3|\}} ^{|\Delta_1|+|\Delta_2|}\{\Delta_1\Delta_2\Delta_3\}_{n}$$
where
$$\{\Delta_1\Delta_2\Delta_3\}_n=\left\{\begin{array}{lr}
\rho_{(n-|\Delta_1|)}(\Delta_1)
\circ \rho_{(n-|\Delta_2|)}(\Delta_2)\circ \rho_{(n-|\Delta_3|)}
(\Delta_3)&\hbox{for }n=\max\{|\Delta_1|,|\Delta_2|,|\Delta_3|\}\\
\\
\rho_{(n-|\Delta_1|)}(\Delta_1)
\circ \rho_{(n-|\Delta_2|)}(\Delta_2)\circ \rho_{(n-|\Delta_3|)}(\Delta_3)-\\
-\sum\limits_ {k=\max\{|\Delta_1|,|\Delta_2|,|\Delta_3|\}}^{n-1}
\rho_{(n-k)}(\{\Delta_1\Delta_2\Delta_3\}_k)&\hbox{for }n>\max\{|\Delta_1|,
|\Delta_2|,|\Delta_3|\}\end{array}\right.$$

\end{proof}

\begin{theorem} {\rm \label{t2} The algebras $A_{\infty}$ and $A_{\infty}^{\circ}$
are naturally isomorphic.}
\end{theorem}
\begin{proof} Product of the formal sums $G_{\infty}(\Delta_1)$ and
$G_{\infty}(\Delta_2)$ is a finite sum of the formal sums of the type
$G_{\infty}(\Delta)$, where $\max\{|\Delta_1|,|\Delta_2|\}\leq |\Delta| \leq
|\Delta_1|+|\Delta_2|$. If
$|\Delta_1|=|\Delta_2|=|\Delta|=n$,
then $G_{\infty}(\Delta)=\sum_g C^{\Delta}_{\Delta_1,\Delta_2}gG_n(\Delta)g^{-1}$,
where the sum goes over all finite permutations $g$ of the set $\mathbb{N}$,
which do not preserve $\{1,2,\dots,n\}$. By the same reason, the formal sum
$G_{\infty}(\Delta)$ for
$|\Delta_1|+1=|\Delta_2|+1=|\Delta|\in\mathcal{A}_n$ is equal to
$\sum_g C^{\Delta}_{\rho_1({\Delta}_1),\rho_1({\Delta}_2)}gG_n(\Delta)g^{-1}$
minus $\rho_1(G_{\infty}(\hat{\Delta}))$, where $\Delta=\rho_1(\hat{\Delta})$.
Similar arguments prove the statement of the theorem for all
$\Delta_1$, $\Delta_2$ of coinciding degrees. If, however,
$|\Delta_1|<|\Delta_2|=|\Delta|$, then the term $G_{\infty}(\Delta)$,
for the product of $\Delta_1$ and $\Delta_2$, coincides with the term
$G_{\infty}(\Delta)$, for the product of $\tilde{\Delta}_1$ and $\Delta_2$,
where $\tilde{\Delta}_1=\rho_{|\Delta_2|-|\Delta_1|}(\Delta_1)$.
\end{proof}

\section{{\bf Differential representation of the algebra $U(gl(\infty))$}\label{s2}}

Consider the set of formal differential operators $$D_{ab}=
\sum\limits_{e\in\{1,\dots,N\}}X_{ae}\frac{\partial}{\partial X_{be}}$$ of the Miwa variables
$\{X_{ij}|i,j\leq N\}$ Multiplication of operators is given by the rule
$$D_{ab} D_{cd}=\sum\limits_{e_1,e_2\in\mathbb{N}}
X_{ae_1}X_{ce_2}\frac{\partial}{\partial X_{be_1}}\frac{\partial}{\partial
X_{de_2}}+\delta_{bc}\sum\limits_{e\in\mathbb{N}}
X_{ae}\frac{\partial}{\partial X_{de}}$$ Commutation relations for the operators
$D_{ab}$ coincide with the commutation relations for the generators of the matrix algebra.
Hence, the operators $D_{ab}$ give rise to the algebra $U(N)$ naturally isomorphic to
the universal enveloping algebra $U(gl(N))$.

In the limit $N\rightarrow\infty$ there emerges the algebra $U_{\infty}$ of the formal
differential operators which are finite or countable sums of operators of the form
$$:D_{a_1b_1}\cdots D_{a_nb_n}:=\sum\limits_{e_1,...,e_n\in\mathbb{N}}
X_{a_1e_1}\cdots X_{a_ne_n}\frac{\partial}{\partial X_{b_1e_1}}\cdots
\frac{\partial}{\partial X_{b_ne_n}}$$
We call the number $|\mathcal{U}|=n$ \textit{degree of the operator} $\mathcal{U}$.
Linear combinations of the operators of the same degree are called
\textit{homogeneous operators}.

Thus, the vector space $U_{\infty}$ is decomposed into the direct sum
$U_{\infty}=\sum\limits_ {n\in\mathbb{N}}U_n$ of the subspaces of homogeneous
operators of degree $n$. Consider the projection
$pr_n: U_{\infty}\rightarrow U_n$, preserving the operators of degree $n$ and
mapping to zero all other homogeneous operators.

Introduce on $U_n$ a multiplication $"\circ"$ by the formula
$\mathcal{U}_1\circ \mathcal{U}_2=pr_n(\mathcal{U}_1\ \mathcal{U}_2)$.
This multiplication turns $U_n$ into an associative algebra of the differential
operators $U_n^{\circ}$.

Consider an embedding of the vector spaces
$$\varrho_k:U_n\rightarrow U_{n+k}$$
where $$\varrho_k(:D_{a_1b_1}\cdots D_{a_nb_n}:)=
\frac{1}{k!}\sum\limits_ {c_1,...,c_k\in\mathbb{N}}:D_{c_1c_1}
\cdots D_{c_kc_k}D_{a_1b_1}\cdots D_{a_nb_n}:$$
The operators $\mathcal{U}$ and $\varrho_k(\mathcal{U})$ acts similarly
on the monomials $X$ of degree $n+k$ of the Miwa variables $\{X_{i,j}\}$.

One immediately checks the following claim:

\begin{theorem}\label{t3}  {\rm There is an equality

$$\mathcal{U}_1\mathcal{U}_2=\sum\limits_ {n=\max\{|\mathcal{U}_1|,
|\mathcal{U}_2|\}} ^{|\mathcal{U}_1|+|\mathcal{U}_2| }
\{\mathcal{U}_1\mathcal{U}_2\}_{n}$$ where $$\{\mathcal{U}_1
\mathcal{U}_2\}_n = \left\{\begin{array}{lr}
\varrho_{(n-|\mathcal{U}_1|)}(\mathcal{U}_1)\circ
\varrho_{(n-|\mathcal{U}_2|)}(\mathcal{U}_2)&\hbox{for }
n=\max\{|\mathcal{U}_1|,|\mathcal{U}_2|\}\\
\\
\varrho_{(n-|\mathcal{U}_1|)}(\mathcal{U}_1)\circ
\varrho_{(n-|\mathcal{U}_2|)}(\mathcal{U}_2)-\sum\limits_
{k=\max\{|\mathcal{U}_1|,|\mathcal{U}_2|\}}^{n-1}
\varrho_{(n-k)}(\{\mathcal{U}_1 \mathcal{U}_2\}_k)\\&\hbox{for }
n>\max\{|\mathcal{U}_1|,|\mathcal{U}_2|\}\end{array}\right.$$}
\end{theorem}

\begin{example}\label{ex2.1} {\rm Put $\mathcal{U}_1=\sum\limits_{e\in\mathbb{N}}
X_{a_1e}\frac{\partial}{\partial X_{b_1e}}$,
$\mathcal{U}_2=\sum\limits_{e\in\mathbb{N}}
X_{a_2e}\frac{\partial}{\partial X_{b_2e}}$.\\
Then $\mathcal{U}_1 \mathcal{U}_2= \delta_{b_1,a_2}
\sum\limits_{e\in\mathbb{N}}X_{a_1e}\frac{\partial}
{\partial X_{b_2e}}+ \sum\limits_{e_1,e_2\in\mathbb{N}}
X_{a_1e_1}X_{a_2e_2}\frac{\partial}{\partial X_{b_1e_1}}
\frac{\partial}{\partial X_{b_2e_2}}$.

\noindent
On the other hand, $\{\mathcal{U}_1 \mathcal{U}_2\}_1=\mathcal{U}_1\circ\mathcal{U}_2=
\delta_{b_1,a_2}\sum\limits_{e\in\mathbb{N}}X_{a_1e}\frac{\partial}{\partial X_{b_2e}}$
and\\ $\varrho_1(\{\mathcal{U}_1 \mathcal{U}_2\}_1)= \delta_{b_1,a_2}
 \sum\limits_{c\in\mathbb{N}} \sum\limits_{e,f\in\mathbb{N}}X_{cf}X_{a_1e}
\frac{\partial}{\partial X_{cf}}\frac{\partial}{\partial X_{b_2e}}$.

\noindent
Besides, $\varrho_1(\mathcal{U}_1)=\sum\limits_{c_1\in\mathbb{N}}
\sum\limits_{e_1,e\in\mathbb{N}}
X_{c_1f_1}X_{a_1e_1}\frac{\partial}{\partial X_{c_1f_1}}
\frac{\partial}{\partial X_{b_1e_1}}\ $ and\\ $\varrho_1(\mathcal{U}_2)=
\sum\limits_{c_2\in\mathbb{N}} \sum\limits_{e_2,e\in\mathbb{N}}
X_{c_2f_2}X_{a_2e_2}\frac{\partial}{\partial X_{c_2f_2}}
\frac{\partial}{\partial X_{b_2e_1}}$. Hence, \\$\varrho_1(\mathcal{U}_1)
\circ\varrho_1(\mathcal{U}_2)= \sum\limits_{e_1,e_2\in\mathbb{N}}
X_{a_1e_1}X_{a_2e_2}\frac{\partial}{\partial X_{b_1e_1}}
\frac{\partial}{\partial X_{b_2e_2}}+\delta_{b_1,a_2}\sum\limits_{c\in\mathbb{N}}
 \sum\limits_{e,f\in\mathbb{N}}X_{cf}X_{a_1e}
\frac{\partial}{\partial X_{cf}}\frac{\partial}{\partial X_{b_2e}}$.

\noindent
Thus, $\{\mathcal{U}_1 \mathcal{U}_2\}_2=\sum\limits_{e_1,e_2\in\mathbb{N}}
X_{a_1e_1}X_{a_2e_2}\frac{\partial}{\partial X_{b_1e_1}}
\frac{\partial}{\partial X_{b_2e_2}}$.}
\end{example}

\begin{example}\label{ex2.2} {\rm
Put $\mathcal{U}_1=\sum\limits_{e_1^1\in\mathbb{N}}
X_{a_1^1e_1^1}\frac{\partial}{\partial X_{b^1_1e_1^1}}$, $\mathcal{U}_2=
\sum\limits_{e_2^1,e_2^2\in\mathbb{N}}
X_{a_2^1e_2^1}X_{a_2^2e_2^2}\frac{\partial}{\partial X_{b_2^1e_2^1}}
\frac{\partial}{\partial X_{b_2^2e_2^2}}$.

\noindent
Then $\mathcal{U}_1 \mathcal{U}_2= \delta_{b_1^1,a_2^1}\sum\limits_{e_1^1,e_2^2
\in\mathbb{N}}
X_{a_1^1e_1^1}X_{a_2^2e_2^2}
\frac{\partial}{\partial X_{b_2^1e_1^1}}\frac{\partial}{\partial X_{b_2^2e_2^2}}+
\delta_{b_1^1,a_2^2}\sum\limits_{e_1^1,e_2^1\in\mathbb{N}}
X_{a_1^1e_1^1}X_{a_2^1e_2^1}
\frac{\partial}{\partial X_{b_2^2e_1^1}}\frac{\partial}{\partial X_{b_2^1e_2^1}}+$
$\sum\limits_{e_1^1,e_2^1,e_2^2\in\mathbb{N}}
X_{a_1^1e_1^1}X_{a_2^1e_2^1}X_{a_2^2e_2^2}
\frac{\partial}{\partial X_{b_1^1e_1^1}}
\frac{\partial}{\partial X_{b_2^1e_2^1}}
\frac{\partial}{\partial X_{b_2^2e_2^2}}$.
On the other hand, $\{\mathcal{U}_1 \mathcal{U}_2\}_2=\varrho_1(\mathcal{U}_1)
\circ\mathcal{U}_2=$\\

\noindent
$(\sum\limits_{c_1^1\in\mathbb{N}} \sum\limits_{f_1^1,e_1^1\in\mathbb{N}}
X_{c_1^1f_1^1}X_{a_1^1e_1^1}\frac{\partial}{\partial X_{c_1^1f_1^1}}
\frac{\partial}{\partial X_{b_1^1e_1^1}})\circ
(\sum\limits_{e_2^1,e_2^2\in\mathbb{N}}
X_{a_2^1e_2^1}X_{a_2^2e_2^2}\frac{\partial}{\partial X_{b_2^1e_2^1}}
\frac{\partial}{\partial X_{b_2^2e_2^2}})=$\\
$\delta_{b_1^1,a_2^1}\sum\limits_{e_2^1,e_2^2\in\mathbb{N}}
X_{a_1^1e_2^1}X_{a_2^2e_2^2}\frac{\partial}{\partial X_{b_2^1e_2^1}}
\frac{\partial}{\partial X_{b_2^2e_2^2}}+ \delta_{b_1^1,a_2^2}\sum\limits_{e_2^1,e_2^2\in\mathbb{N}}
X_{a_1^1e_2^2}X_{a_2^1e_2^1}\frac{\partial}{\partial X_{b_2^1e_2^1}}
\frac{\partial}{\partial X_{b_2^2e_2^2}}$
and\\
$\varrho_1(\{\mathcal{U}_1 \mathcal{U}_2\}_2)=
\sum\limits_{c_1^1\in\mathbb{N}}\sum\limits_{f_1^1,e_1^1\in\mathbb{N}}$
$\delta_{b_1^1,a_2^1}\sum\limits_{e_2^1,e_2^2\in\mathbb{N}}
X_{c_1^1f_1^1}X_{a_1^1e_2^1}X_{a_2^2e_2^2}
\frac{\partial}{\partial X_{c_1^1f_1^1}}
\frac{\partial}{\partial X_{b_2^1e_2^1}}
\frac{\partial}{\partial X_{b_2^2e_2^2}}+$
$\delta_{b_1^1,a_2^2}\sum\limits_{c_1^1\in\mathbb{N}}
\sum\limits_{e_2^1,e_2^2\in\mathbb{N}}
X_{c_1^1f_1^1}X_{a_1^1e_2^2}X_{a_2^1e_2^1}
\frac{\partial}{\partial X_{c_1^1f_1^1}}
\frac{\partial}{\partial X_{b_2^1e_2^1}}
\frac{\partial}{\partial X_{b_2^2e_2^2}}$.
Besides,\\
$\varrho_2(\mathcal{U}_1)\circ\varrho_1(\mathcal{U}_2)=$
$\sum\limits_{e_1^1,e_2^1,e_2^2\in\mathbb{N}}
X_{a_1^1e_1^1}X_{a_2^1e_2^1}X_{a_2^2e_2^2}
\frac{\partial}{\partial X_{b_1^1e_1^1}}
\frac{\partial}{\partial X_{b_2^1e_2^1}}
\frac{\partial}{\partial X_{b_2^2e_2^2}}+\\$
$\sum\limits_{c_1^1\in\mathbb{N}}\sum\limits_{f_1^1,e_1^1\in\mathbb{N}}
\delta_{b_1^1,a_2^1}\sum\limits_{e_2^1,e_2^2\in\mathbb{N}}
X_{c_1^1f_1^1}X_{a_1^1e_2^1}X_{a_2^2e_2^2}
\frac{\partial}{\partial X_{c_1^1f_1^1}}
\frac{\partial}{\partial X_{b_2^1e_2^1}}
\frac{\partial}{\partial X_{b_2^2e_2^2}}+\\$
$\delta_{b_1^1,a_2^2}\sum\limits_{c_1^1\in\mathbb{N}}
\sum\limits_{e_2^1,e_2^2\in\mathbb{N}}
X_{c_1^1f_1^1}X_{a_1^1e_2^2}X_{a_2^1e_2^1}
\frac{\partial}{\partial X_{c_1^1f_1^1}}
\frac{\partial}{\partial X_{b_2^1e_2^1}}
\frac{\partial}{\partial X_{b_2^2e_2^2}}$.

Thus, $\mathcal{U}_1 \mathcal{U}_2=\{\mathcal{U}_1 \mathcal{U}_2\}_1+
\{\mathcal{U}_1 \mathcal{U}_2\}_2$.}
\end{example}

\section{{\bf Algebra $\mathcal{W}_{\infty}$ of the differential operators}\label{s3}}

Associate with the Young diagram $\Delta=[\mu_1,\mu_2,\dots,\mu_l]$
with the ordered row lengths $\mu_1\geq\mu_2\geq\dots\geq\mu_l$ the numbers $m_k=m_k(\Delta)=
|\{i|\mu_i=k\}|$
and $\kappa(\Delta)=(\prod\limits_{k}m_k!k^{m_k})^{-1}$.

Associate with the Young diagram $\Delta$ the operator $\mathcal{W}(\Delta)=
\kappa(\Delta)\prod\limits_{k}
:D_k^{m_k}:\ \in U_{\infty}$.

\begin{example} \label{ex3.1} {\rm $\mathcal{W}([1])=
\sum\limits_{a\in\mathbb{N}} :D_{aa}:$ \quad
$\mathcal{W}([2])=\frac{1}{2}\sum\limits_{a,b\in\mathbb{N}} :D_{ab}D_{ba}:$}
\end{example}

Denote through $\mathcal{W}^{\circ}_n$ the vector space generated by the operators of the form
$\mathcal{W}(\Delta)$, where $|\Delta|=n$.

\begin{lemma}\label{l3.1} {\rm The operation $"\circ"$ provides
the structure of algebra on $\mathcal{W}^{\circ}_n$.
The correspondence
$\Delta\mapsto\mathcal{W}(\Delta)$ gives rise to the isomorphism of algebras
$\psi_n: A_n^{\circ}\rightarrow \mathcal{W}_n^{\circ}$.}
\end{lemma}
\begin{proof} Associate with each permutation $g\in S_n$ the operator
$\mathcal{W}(g)=\kappa(\Delta(g))\sum\limits_{a_1,...,a_n\in\mathbb{N}}
:D_{a_1a_{g(1)}}\cdots D_{a_na_{g(n)}}:$ Then $\mathcal{W}(\Delta(g))=
\mathcal{W}(g)$. Hence, the claim of the lemma follows from the equality
$\mathcal{W}(\Delta(g_1)\circ\Delta(g_2))=
\mathcal{W}(g_1)\circ \mathcal{W}(g_2)$ для $g_1,g_2\in S_n$.
\end{proof}

\begin{lemma}\label{l3.2} {\rm The embedding $\varrho_k:U_n\rightarrow U_{n+k}$ gives rise to
the embedding $\varrho_k:\mathcal{W}^{\circ}_n\rightarrow
\mathcal{W}^{\circ}_{n+k}$, where $\varrho_k\psi_n(\Delta)=\psi_{n+k}\rho_k(\Delta)$
at $\Delta\in A_n$.}
\end{lemma}

\begin{proof} The map $\varrho_k:U_n\rightarrow U_{n+k}$
gives rise to the correspondence
$\varrho_k(\mathcal{W}(\Delta))=\frac{1}{k!}\mathcal{W}
({\Delta}^k)$.

In accordance with our definitions, $\psi_n(\Delta)=
\kappa(\Delta)
\sum\limits_{a_1,...,a_n\in\mathbb{N}} :D_{a_1a_{g(1)}}\cdots D_{a_na_{g(n)}}:$ and
$\varrho_k\psi_n(\Delta)=
\frac{1}{k!}\kappa(\Delta)
\sum\limits_{c_1,...,c_k,a_1,...,a_n\in\mathbb{N}} :D_{c_1c_1}\cdots D_{c_kc_k}
D_{a_1a_{g(1)}}\cdots D_{a_na_{g(n)}}:$

On the other hand, $\rho_k(\Delta)=\frac{(m_1+k)!}{m_1!k!}{\Delta}^k$ and \\
$\psi_{n+k}\rho_k(\Delta)= \frac{(m_1+k)!}{m_1!k!}
\kappa(\Delta)
(\frac{(m_1+k)!}{m_1!})^{-1}\sum\limits_{c_1,...,c_k,a_1,...,a_n\in\mathbb{N}}
:D_{c_1c_1}\cdots D_{c_kc_k} D_{a_1a_{g(1)}}\cdots D_{a_na_{g(n)}}:=
\\ \frac{1}{k!}(\prod\limits_{k}m_k!k^{m_k})^{-1}
\sum\limits_{c_1,...,c_k,a_1,...,a_n\in\mathbb{N}} :D_{c_1c_1}\cdots D_{c_kc_k}
D_{a_1a_{g(1)}}\cdots D_{a_na_{g(n)}}:$
\end{proof}

\begin{example}\label{ex3.2} {\rm Put $\Delta=[1]$.
Then $\rho_k(\Delta)=(k+1)[1,1,\dots,1]$ and \\
$\psi_{k+1}(\rho_k(\Delta))=
\psi_{k+1}((k+1)[1,1,\dots,1])=\frac{(k+1)}{(k+1)!} \sum\limits_{c_1,...,c_{k+1}
\in\mathbb{N}}
:D_{c_1c_1}\cdots D_{c_nc_n}:$

On the other hand, $\psi_{1}([1])=\sum\limits_{e\in\mathbb{N}}
:D_{cc}:$ and  \\
$\varrho_{k}(\sum\limits_{e\in\mathbb{N}}
:D_{cc}:)=\frac{1}{k!} \sum\limits_{e,e_1,...,e_k\in\mathbb{N}}
\sum\limits_{e\in\mathbb{N}}
:D_{cc}D_{c_1c_1}\cdots D_{c_nc_n}:$}
\end{example}

\begin{example}\label{ex3.3} {\rm Put $\Delta=[2]$.
$\psi_{2}([2])=
\frac{1}{2}\sum\limits_{a,b\in\mathbb{N}} :D_{ab}D_{ba}:$
and \\
$\varrho_{k}(\frac{1}{2}\sum\limits_{a,b\in\mathbb{N}} :D_{ab}D_{ba}:)=$
$\frac{1}{2k!}\sum\limits_{a,b,e_1,...,e_k\in\mathbb{N}} :D_{ab}D_{ba}D_{c_1c_1}
\cdots D_{c_nc_n}:$

On the other hand, $\rho_k(\Delta)= [2,1,\dots,1]$ and \\
$\psi_{k+2}(\rho_k(\Delta))=\psi_{k+2}([2,1,\dots,1])=
\frac{1}{2k!}\sum\limits_{a,b,e_1,...,e_k\in\mathbb{N}} :D_{ab}D_{ba}D_{c_1c_1}
\cdots D_{c_nc_n}:$}
\end{example}

Denote through $\mathcal{W}_{\infty}\subset U_{\infty}$ the subalgebra generated
by the differential operators $\mathcal{W}(\Delta)$. Confronting Theorems
\ref{t1}, \ref{t2}, \ref{t3} with Lemmas \ref{l3.1},
\ref{l3.2}, one obtains

\begin{theorem}\label{t4} {\rm The isomorphisms $\psi_n$ gives rise to the isomorphism
of the algebras
$\psi: A_{\infty}\rightarrow W_{\infty}$.}
\end{theorem}

\section{{\bf Schur functions\label{s4}}}

The Schur function of $n$ variables corresponds to the Young diagram
$R=\{R_1\geq R_2\geq
\dots\geq R_m>0\}$, where $n\geq m$. It is defined by the formula
$$s_{R}(x_1,\dots,x_n)=\frac{\det[x_i^{R_j+n-j}]_{1\leq i, j
\leq n}}{\det[x_i^{n-j}]_{1\leq i, j\leq n}},$$
where $R_i=0$ at $m<i\leq n$. The property of stability
$s_{R}(x_1,\dots,x_n,0)=s_{R}(x_1,\dots,x_n)$ allows one to define
$s_{R}$ on an arbitrary finite set of variables.

Define the Schur functions on finite matrices $X\in gl(n)\subset
gl(\infty)$ by the formula $s_{R}(X)= s_{R}(x_1,\dots,x_n)$, where
$x_1,\dots,x_n$ are the eigenvalues of the matrix $X$. The functions
$s_{R}(X)$ form a basis in the vector space $\textsf{P}$ of all symmetric
polynomials of the Miwa variables.

Polynomials of the variables $p_i=\tr X^i$ form another natural basis
of the space $\textsf{P}$,
$$
\tilde s_R(p)=\det [P_{R_i+j-i}(p)]_{1\leq i, j
\leq n}\,,\ \ \ \ \ \ \ \exp\big(\sum_k p_kx^k\big)\equiv\sum_iP_i(p)x^i
$$
where $p=(p_1,p_2\dots)$. Then $s_{R}(X)=\tilde s_R(p)$  \cite{M}. Associate with the Young
diagram $\Delta$ the monomial $p(\Delta)=
\kappa(\Delta)p_1^{m_1(\Delta)} p_2^{m_2(\Delta)}\dots
p_n^{m_n(\Delta)}$.

Put $\Delta^k=[\Delta,\underbrace{1,...,1}_k]$. Let $\dim R$ be the dimension
of representation of the symmetric group $S_{|R|}$ corresponding to the diagram
$R$, and $\chi_R(X_{\Delta})$ be the value of character of this representation on the
element of the cyclic type $\Delta^{|R|-|\Delta|}$. Put
$d_R=\frac{\dim R}{|R|!}=\frac{ \prod\limits_{i < j = 1}^{|R|}
\left( \mu_i -\mu_j - i + j \right) } {\prod\limits_{i = 1}^{|R|}
\left(\mu_i +|R| - i\right)!}$ and $m_1=m_1(\Delta)$.

Denote through $\varphi_R(\Delta)$ the function which is equal to $\frac{\kappa(\Delta)}
{d_R m_1!(|R|-|\Delta|-m_1)!}\chi_R(X_{\Delta})$
at $|R|-|\Delta|\geq m_1$ and 0 otherwise. Then $\varphi_R(\Delta^k)=
\frac{m_1!k!}{(m_1+k)!}\varphi_R(\Delta)$ at $k=|R|-|\Delta|$.

\begin{theorem}\label{t5} {\rm The functions $s_{R}(X)$ are the eigenfunctions of the operators
$\mathcal{W}(\Delta)$. They form a complete system of eigenfunctions
for the restrictions
$\textsf{W}(\Delta)=\mathcal{W}(\Delta)|_{\textsf{P}}$ and
$\textsf{W}(\Delta)(s_R)=\varphi_R(\Delta)s_R$.}
\end{theorem}

\begin{proof} Consider the regular representation of the algebra
$U(gl(N))$ in the algebra of polynomial functions of matrix elements of
$gl(N)$. The center $Z(U(gl(N)))$ preserves the vector subspace $\textsf{P}$. The algebra
$Z(U(gl(N)))$ is additively generated by the operators $T(\Delta)$ associated with the
Young diagrams $\Delta$, and $T(\Delta)(f)=
\mathcal{W}(\Delta)(f)$ for $f\in \textsf{P}$. Besides, in accordance with the Weyl theorem
\cite{Zh}, the Schur functions $s_R(X)$ form a complete system of the eigenfunctions of the
operators $T(\Delta)$. Taking the limit $N\rightarrow\infty$, one finds that the Schur functions
form a complete system of the eigenfunctions of the operators
$\textsf{W}(\Delta)$.

Now we find the eigenvalues of the operators. In accordance with
\cite[s.1.7]{M},
$p(\Delta)=\sum\limits_{R:|R|= |\Delta|}d_R\varphi_R(\Delta)\tilde s_R$.
Hence,
$$p(\Delta)e^{p_1}=
\sum\limits_{k=0}^{\infty}p(\Delta)\frac{p_1^k}{k!} \sum\limits_{k=0}^{\infty}
\frac{(m_1+k)!}{m_1!k!} p(\Delta^k)=\sum\limits_{k=0}^{\infty}\sum\limits_{|R|=
|\Delta|+k}\frac{(m_1+k)!}{m_1!k!}d_R\varphi(\Delta^k)\tilde s_R=$$
$$=
\sum\limits_{k=0}^{\infty}\sum\limits_{|R|=|\Delta|+k}
d_R\varphi(\Delta)\tilde s_R=\sum\limits_R d_R\varphi(\Delta)\tilde s_R$$

On the other hand, in accordance with \cite[s.1.4, example 3]{M}, $e^{p_1} =
\sum\limits_R  d_R\tilde s_R(p)$, hence, $p(\Delta)e^{p_1}=
\textsf{W}(\Delta)(e^{p_1})=\sum\limits_R d_R\textsf{W}(\tilde s_R)$.
Thus, $\sum\limits_R d_R\textsf{W}(s_R)=\sum\limits_R d_R\varphi(\Delta)s_R$.

We have already proved that $s_R$ form a complete system of the eigenfunctions
of the operator
$\mathcal{W}$. Therefore, the last equality implies $\textsf{W}(\Delta)(s_R)=
\varphi_R(\Delta)s_R$.
\end{proof}

\begin{corollary} {\rm The values of $\varphi_R(\Delta)$ are related by the formula
$$\varphi_R(\Delta_1)\varphi_R(\Delta_2)= \sum\limits_{\Delta}
C^{\Delta}_{\Delta_1\Delta_2}\varphi_R(\Delta)$$
where $C^{\Delta}_{\Delta_1\Delta_2}$ are the structure constants of the algebra $A_{\infty}$,
which are obtained in s.\ref{s1}.}
\end{corollary}

\section{{\bf First few $\textsf{W}$-operators\label{s5}}}

Represent now the operators $\textsf{W}(\Delta)$ as differential
operators of the variables $\{p_k\}$. Then,
$$ D_{ab} F(p) = X_{ac}\frac{\partial}{\partial X_{bc}} F(p) =
\sum_{k=1}^\infty k(X^k)_{ab} \frac{\partial F(p)}{\partial p_k}$$

Using the relation

$$D_{a'b'}(X^k)_{ab} =
X_{a'c'}\frac{\partial}{\partial X_{b'c'}}(X^k)_{ab}
= \sum_{j=0}^{k-1} X_{a'c'} (X^j)_{ab'}(X^{k-j-1})_{c'b}
= \sum_{j=0}^{k-1} (X^j)_{ab'} (X^{k-j})_{a'b},$$ one obtains

$$D_{a'b'} D_{ab} F(p) = \sum_{k,l=1}^\infty
kl(X^l)_{a'b'}(X^k)_{ab}\frac{\partial^2F(p)}{\partial p_k p_l} +
\sum_{k=1}^\infty\sum_{j=0}^{k-1}k (X^j)_{ab'}(X^{k-j})_{a'b}
\frac{\partial F(p)}{\partial p_k}$$

Thus,

$$: D_{a'b'}D_{ab}:\, F(p) =
\sum_k\left(k\sum_{j=1}^{k-1} (X^j)_{ab'} (X^{k-j})_{a'b}\right)
{\partial F(p)\over\partial p_k}+
\sum_{k,l} kl(X^k)_{ab}(X^l)_{a'b'}{\partial ^2 F(p)\over\partial p_k\partial p_l}.$$

This relation allows one to find all the operators $\textsf{W}$. In particular,
$$
\textsf{W}([1]) = \tr \hat D= \sum_{k=1} kp_k\frac{\partial}{\partial p_k}
$$
$$
\textsf{W}([2]) ={1\over 2} \, :D^2\,:\ =
\frac{1}{2}\sum_{a,b=1}^\infty
\left( (a+b)p_ap_b\frac{\partial}{\partial p_{a+b}} +
abp_{a+b}\frac{\partial^2}{\partial p_a\partial p_b}\right)
$$
$$
\textsf{W}([1,1]) = \frac{1}{2!}\, :(\tr D)^2\,:\ =
\frac{1}{2}\left(\sum_{a=1}^\infty
 a(a-1)p_a\frac{\partial}{\partial p_{a}} +\sum_{a,b=1}^\infty
abp_{a}p_b\frac{\partial^2}{\partial p_a\partial p_b}\right)
$$
$$
\textsf{W}([3]) = \frac{1}{3}\, :\tr D^3\,:\ =
\frac{1}{3}\sum_{a,b,c\geq 1}^\infty
abcp_{a+b+c} \frac{\partial^3}{\partial p_a\partial p_b\partial p_c}
+ \frac{1}{2}\sum_{a+b=c+d} cd\left(1-\delta_{ac}\delta_{bd}\right)
p_ap_b\frac{\partial^2}{\partial p_c\partial p_d} +$$
$$+ \frac{1}{3} \sum_{a,b,c\geq 1}
(a+b+c)\left(p_ap_bp_c + p_{a+b+c}\right)\frac{\partial}{\partial p_{a+b+c}}
$$
$$
\textsf{W}([2,1]) = \frac{1}{2}\, :\tr D^2\,\tr D \,:\ =
{1\over 2}\sum_{a,b\ge 1}(a+b)(a+b-2)p_ap_{b}{\partial\over\partial p_{a+b}}\,+
{1\over 2}\sum_{a,b\ge 1}ab(a+b-2)p_{a+b}{\partial ^2\over\partial p_a\partial p_b}\,
+$$
$$+\frac{1}{2}\sum_{a,b,c\ge 1} (a+b)cp_ap_bp_c\frac{\partial^2}{\partial p_{a+b}
\partial p_c}
+{1\over 2}\sum_{a,b,c\ge 1}abcp_ap_{b+c}{\partial ^3\over\partial p_a\partial p_
b\partial p_c}
$$
$$
\textsf{W}([1,1,1]) = \frac{1}{3!}\, :(\tr D)^3\,:\ =
{1\over 6}\sum_{a\ge 1} a(a-1)(a-2)p_a{\partial\over\partial p_a}\,+$$
$$+ {1\over 4}\sum_{a,b} ab(a+b-2)p_ap_b\frac{\partial^2}{\partial p_a\partial p_b}\,
+{1\over 6}\sum_{a,b,c\ge 1}abcp_ap_bp_c{\partial ^3\over\partial p_a\partial p_b
\partial p_c}
$$

\section{{\bf Hurwitz numbers\label{s7}}}

Each holomorphic morphism of degree $n$ of Riemann surfaces
$f:\widetilde{\Omega}\rightarrow\Omega$ associates with the point
$s\in\Omega$ a local invariant: the Young diagram $\Delta(f,s)$ of degree $n$
with the row lengths being equal to degrees of the map $f$
at the points of complete pre-image
$f^{-1}(s)=
\{s^1,\dots,s^k\}$. More than 100 years ago Hurwitz \cite{H} formulated
a problem of calculating {\it the Hurwitz numbers}
$$H((s_1,\Delta_1),\dots,(s_k,\Delta_k)|\Omega)=\sum_{f\in
Cov_n(\Omega, \{\alpha_1, \dots,\alpha_s\})} \frac{1}{|\Aut(f)|}$$
for an arbitrary set $\{\Delta_1,\dots,\Delta_k\}$
of Young diagrams of degree $n$.
Here $|\Aut(f)|$ is the order of automorphism group of the map $f$, and
$Cov_n(\Omega, \{\alpha_1,\dots,\alpha_s\})$ is a set of classes of the
biholomorphic equivalence of the holomorphic morphisms
$f:\widetilde{\Omega}\rightarrow\Omega$ with the set of critical values
$s_1,\dots,s_k\in\Omega$ and the local invariants
$\alpha(f,s_i)=\alpha_i$.

This number depends only on the genus $g(\Omega)$ of the surface $\Omega$
and the diagrams $\Delta_1,\dots,\Delta_k$. We define
$<\Delta_1,\dots,\Delta_k>_{g(\Omega)}\ =
 H((s_1,\Delta_1),\dots,(s_k,\Delta_k)|\Omega)$.
The Hurwitz numbers of any genus are easily expressed through those
at genus zero,
$<\Delta_1,\dots,\Delta_k>\ =\ <\Delta_1,\dots,\Delta_k>_0$, \cite{AN}.

A defining property of the Hurwitz numbers is \textit{the associativity
relation}
$$<\Delta_1,\dots,\Delta_k>\ =
\sum\limits_{\Upsilon\in\mathcal{A}_n}
<\Delta_1,\dots,\Delta_r,\Upsilon>|\Aut(\Upsilon)|
<\Upsilon,\Delta_{r+1},\dots,\Delta_k>$$
The Hurwitz numbers of coverings with three critical values are related to the
structure constants of the algebra $A_n$ by the formula
$<\Delta_1,\Delta_2,\Delta_3>\ =C_{\Delta_1,\Delta_2}^{\Delta_3} |\Aut(\Delta_3)|^{-1}$.
Arbitrary Hurwitz numbers are expressed through these simplest Hurwitz numbers
by the formula
$$<\Delta_1,\dots,\Delta_k>\ =
\sum\limits_{\Upsilon_1,\dots,\Upsilon_{k-1}\in\mathcal{A}_n}
<\Delta_1,\Delta_2,\Upsilon_1>|\Aut(\Upsilon_1)|
<\Upsilon_1,\Delta_3,\Upsilon_2>\times$$
$$\times|\Aut(\Upsilon_2)|\dots|\Aut(\Upsilon_{k-1})|<\Upsilon_{k-1},
\Delta_{k-1},\Delta_k>,$$ (see, e.g., \cite{AN}).

The Hurwitz numbers appear in different framewroks: strings and QCD \cite{GT},
mirror symmetry
\cite{D}, theory of singularities \cite{A}, matrix models \cite{KSW,MM},
integrable systems \cite{OP,GKM2,O}, Yang-Mills theory \cite{CMR,GKM2} and the theory
of moduli of curves \cite{KL,K,MM} and other branches of string theory.

Associate with Young diagrams $\Delta_1,\dots,\Delta_k$ and $\Delta$, where
$|\Delta_i|\le |\Delta|$ for all $i$, the numbers
$<(\Delta_1,n_1),\dots,(\Delta_k,n_k)|\Delta>$ equal to the Hurwitz numbers
$<\tilde{\Delta}_1,\dots,\tilde{\Delta}_1,\tilde{\Delta}_2,
\dots,\tilde{\Delta}_2,\dots,\tilde{\Delta}_k,\dots,\tilde{\Delta}_k,\Delta>$,
where the Young diagram $\tilde{\Delta}_i=\rho_{|\Delta|-|\Delta_i|}(\Delta_i)$
is met exactly $n_i$ times. We also put
$<(\Delta_1,n_1),\dots,(\Delta_k,n_k)|\Delta>\ =0$, if $|\Delta_i|>|\Delta|$ at least
for one $i$.

Associate a variable $\beta_{\Delta}$ with each Young diagram $\Delta$ and
consider the generating function for the Hurwitz numbers $$\mathcal{Z}=
\sum\limits_{k=1}^{\infty}
\sum\limits_{\Delta,\Delta_1,\dots,\Delta_k\in\mathcal{A}_{\infty}}
\sum\limits_{n_1,\dots,n_k\in\mathbb{N}}
\frac{\beta_{\Delta_1}^{n_1}\dots\beta_{\Delta_n}^{n_k}}
{n_1!\dots n_k!}<\Delta_1^{n_1},\dots,\Delta_k^{n_k}|\Delta>
p(\Delta).$$

\begin{theorem}\label{t7} {\rm For any Young diagram $\Upsilon$ there is an equality
$$\frac{\partial\mathcal{Z}}{\partial\beta_{\Upsilon}}=
\textsf{W}(\Upsilon)\mathcal{Z}$$}
\end{theorem}

\begin{proof} The claim of the theorem implies a system of relations
between the numbers
$<\Delta_1^{n_1},\dots,\Delta_k^{n_k}|\Delta>$. In accordance with Theorems
\ref{t2} and \ref{t4}, these relations are of the form
$<\Delta_1^{n_1},\dots,\Delta_i^{n_i},\dots,\Delta_k^{n_k}| \Delta>\ =
\ <\Delta_1^{n_1},\dots,\Delta_i^{n_i-1},\dots,\Delta_k^{n_k}|
\Delta\circ\tilde{\Delta}_i>$ and follow from the associativity relation.
\end{proof}

For the Young diagram $\Upsilon=[2]$ and $\beta_{\Upsilon}=0$ at
$\Upsilon\neq [2]$ Theorem \ref{t7} is equivalent to the "cut-and-join"
relation \cite{GJV}. Using the equations with the initial data
${\mathcal{Z}}_0=e^{p_1}$ at all $\beta_\Upsilon=0$ allows one to represent
${\mathcal{Z}}$ as the exponential of the operators
$\textsf{W}(\Upsilon)$ acting on ${\mathcal{Z}}_0$ and calculate this way
any Hurwitz number.

Simplest equations of this kind for the Hurwitz numbers for the
surfaces with boundaries \cite{AN1,AN2} are found in \cite{N}.

\end{document}